\documentclass[11pt]{amsart}

\usepackage[all]{xy}
\usepackage{amssymb,array,xspace}
\usepackage{graphicx}

\newcommand*{\C}{\ensuremath{\mathbb C}\xspace}

\newcommand*{\E}{\ensuremath{\mathcal E}\xspace}

\newcommand*{\Ll}{\ensuremath{\mathcal L}\xspace}

\newcommand*{\Oo}{\ensuremath{\mathcal O}\xspace}
\newcommand*{\PP}{\ensuremath{\mathbb P}\xspace}

\newcommand*{\Q}{\ensuremath{\mathbb Q}\xspace}

\newcommand*{\T}{\ensuremath{\mathcal T}\xspace}
\newcommand*{\U}{\ensuremath{\mathcal U}\xspace}
\newcommand*{\V}{\ensuremath{\mathcal V}\xspace}
\newcommand*{\Z}{\ensuremath{\mathbb Z}\xspace}

\DeclareMathOperator{\Aut}{Aut}

\DeclareMathOperator{\ev}{ev}
\DeclareMathOperator{\Ev}{Ev}

\DeclareMathOperator{\GL}{GL}

\DeclareMathOperator{\Ker}{Ker}

\DeclareMathOperator{\Pic}{Pic}

\DeclareMathOperator{\rank}{rank}

\DeclareMathOperator{\Sym}{Sym}

\let\ph\varphi

\newcommand*{\lst}[3][1]{\ensuremath{#2_{#1}, \ldots, #2_{#3}}\xspace}

\newtheorem{proposition}{Proposition}[section]
\newtheorem{theorem}[proposition]{Theorem}

\newtheorem{lemma}[proposition]{Lemma}

\theoremstyle{remark}
\newtheorem{note}[proposition]{Remark}

\theoremstyle{definition}
\newtheorem{definition}[proposition]{Definition}

\newtheorem{Not}[proposition]{Notation}

\author{Serge Lvovski}
\address{National Research University Higher School of Economics,
  Moscow, Russia\hfil\break
Federal State Institution ``Scientific-Research Institute for
System Analysis of the Russian Academy of Sciences'' (SRISA),
Nakhimovskii pr. 36--1, Moscow, 117218 Russia}
\email{lvovski@gmail.com}

\title[Threefolds with the smallest monodromy group]{On threefolds with
  the smallest nontrivial monodromy group}  

\keywords{Vanishing cycles, monodromy, Picard--Lefschetz theory, adjunction} 

\subjclass{14D05, 14N30} 

\thanks{This study was supported in part by SRISA research project No. 0580-2021-0007
  (Reg. No. 121031300051--3).} 

\begin{document}

\begin{abstract}
Using an adjunction-theoretic result due to A.~J.~Sommese together
with a proposition from SGA7, we obtain a complete list of smooth
threefolds for which the monodromy group acting on $H^2$ of its smooth
hyperplane section is~$\Z/2\Z$. The possibility of such a
classification was announced by F.~L.~Zak in 1991.
\end{abstract}

\maketitle

\section{Introduction}

Suppose that $X\subset\PP^N$ is a smooth projective variety of
dimension~$n$ (throughout the paper, we will assume that the base
field is \C and $X$ is not a linear subspace of $\PP^N$) and that
$Y\subset X$ is its smooth hyperplane section. As $Y$ varies in the
family of smooth hyperplane sections of $X$, a monodromy action on
$H^{n-1}(Y,\Q)$ arises; its image in $\GL(H^{n-1}(Y,\Q))$ will be
called \emph{monodromy group} of~$X$. 

If $\dim X$ is odd, this group is trivial if and only if the dual
variety $X^*\subset (\PP^N)^*$ is not a hypersurface (indeed, if $X^*$
is not a hypersurface, then $\pi_1((\PP^N)^*\setminus X^*)$ is
trivial, and if $X^*$ is a hypersurface and $\dim X$ is odd then the
monodromy group contains reflections in non-zero vanishing cycles;
see~\cite[6.3.3]{Lamotke}).  It is known that, for a smooth
threefold~$X$, its dual $X^*$ is not a hypersurface if and only if $X$
is a $\PP^2$-scroll over a curve (see \cite[Theorem 3.2]{Ein1}). In
this paper we obtain a complete description of the next natural class
of threefolds, that is, of those with monodromy group~$\Z/2\Z$.

\begin{theorem}\label{3folds}
Suppose that $X\subset\PP^N$ is a smooth projective threefold. Then
the monodromy group of $X$ is isomorphic to $\Z/2\Z$ if and only if
$X$ is one of the following varieties:
\begin{enumerate}
\item the quadric $Q\subset\PP^4$;

\item the Veronese variety $v_2(\PP^3)\subset\PP^9$
or its isomorphic projection; 

\item
the blowup of $\PP^3$ at a point, embedded in $\PP^8$
  by the complete linear system $|2H-E|$, where $H$ is the preimage of
  a plane in $\PP^3$ and $E$ is the exceptional divisor, or an isomorphic
  projection of this variety;

\item the Segre variety
$\PP^1\times\PP^1\times\PP^1\subset\PP^7$.
\end{enumerate}
\end{theorem}

In his paper~\cite{Zak91}, F.~L.~Zak announced without proof that,
for a smooth variety $X\subset\PP^N$ of odd dimension $n=2k+1$
the following three assertions are equivalent.
\begin{enumerate}
\item The monodromy group of~$X$ is $\Z/2\Z$;\label{item:Z/2Z}
\item $b_{2k}(Y)=b_{2k}(X)+1$, where $Y$ is a smooth hyperplane
  section.\label{item:dim=1} 
\item The dual variety $X^*\subset(\PP^N)^*$ is a normal
  hypersurface.\label{item:normal}   
\end{enumerate}
Thus, Theorem~\ref{3folds} yields a complete description of, say, smooth
threefolds the dual of which is a normal hypersurface, too.

The equivalence $(\ref{item:Z/2Z})\Leftrightarrow (\ref{item:dim=1})$
is easy (see the proof of Proposition~\ref{c_2=2} below); the proof of
the equivalence $(\ref{item:dim=1})\Leftrightarrow(\ref{item:normal})$
will be published in a separate paper, joint with Zak.

In our proof of Theorem~\ref{3folds} we heavily use the following two
results: a proposition (due to Deligne) from SGA7 (see
Proposition~\ref{Deligne} below) and the main result of Sommese's
paper~\cite{Sommese}. To wit, it turns out that the classification of
smooth threefolds $X\subset\PP^N$ for which $h^0(\omega_X(1))=0$,
which is contained in~\cite{Sommese}, is the same as the
classification of smooth threefolds with finite monodromy group. Using
Deligne's result mentioned above, one can extract from Sommese's list
of such varieties those with~$\Z/2\Z$ monodromy group. Actually, we
obtain a complete description of monodromy groups of varieties on
Sommese's list, with one exception.

In the above mentioned paper~\cite{Zak91} it was announced that a
complete classification of smooth threefolds $X\subset\PP^N$ with
monodromy group $\Z/2\Z$ could be obtained basing on a detailed study
of the second fundamental form of $X$. Be it for better or for worse,
the results of the paper~\cite{Sommese} (i.e., the adjunction
theoretic methods) have lead to success sooner. Yet it would be 
interesting to find a more conceptual proof of Theorem~\ref{3folds},
especially a proof that could be extended to higher (odd)
dimensions.

Besides Theorem~\ref{3folds} we prove the following two results.

First, if $\E$ is a very ample vector bundle of rank~$2$ over a smooth
surface~$S$, $c_2(\E)=r$, and if we vary a section of~$\E$ in the space of
sections with precisely~$r$ zeroes, then this variation induces a
group of permutations of the zero locus of one such section.
It turns out that this group is always the entire symmetric
group~$S_r$ (Proposition~\ref{prop:zeroes-of-sections}).

Second, we show that if $X\subset\PP^N$ is a three-dimensional scroll
over a smooth projective surface, then the curve that is its general
one-dimensional linear section cannot be ``too special''
(Proposition~\ref{linear-section}). 

After $\Z/2\Z$, the next smallest monodromy group of an
odd-dimensional variety is the symmetric group~$S_3$ (it is the Weyl
group of the $A_2$ root system; see
Section~\ref{survey:finite.monodromy}). It seems possible to extract a
list of threefolds with such monodromy from Sommese's list as well; to
that end, the main result of~\cite{Noma2} may be of
help. Conjecturally, the only smooth threefolds with monodromy
group~$S_3$ (``with the $A_2$ monodromy'') are the scroll
$\PP(\Oo_{\PP^2}(1)\oplus \Oo_{\PP^2}(3))$ and the smooth hyperplane
section of the Segre variety $\PP^2\times\PP^2\subset\PP^8$. The
latter threefold is also a scroll over $\PP^2$, namely, the
projectivisation of $\T_{\PP^2}$ (or of $\Omega^1_{\PP^2}$, depending
on the conventions).

\medskip

The paper is organized as follows. In Section~\ref{sec:PLsurvey} we
briefly recall what we need from Picard--Lefschetz theory.  In
Section~\ref{sec:case-by-case} we extract a characterization of
threefolds with finite monodromy group from Sommese's main result
in~\cite{Sommese}. In
Sections~\ref{ntl:subseq:scrolls}--\ref{sec:odds&ends} we describe
vanishing root systems (see the definition in
Section~\ref{survey:finite.monodromy}) for the varieties on Sommese's
list and extract from the list those with $\Z/2\Z$ monodromy.

\subsection*{Acknowledgements}
I am grateful to Fyodor Zak for numerous fruitful discussions of the
subject of this paper. I would like to thank the anonymous referees for really valuable suggestions and corrections.

\section{Notation and conventions}\label{notconv}

The base field will always be the field of complex numbers. 

If $X\subset\PP^N$ is a Zariski closed subset and $x\in X$, then
$T_xX\subset\PP^N$ is the embedded Zariski tangent space to $X$
at~$x$. 

If \E is a vector bundle aka locally free sheaf on $X$, then (unlike
the EGA notation) closed points of the projectivisation $\PP(\E)$ are
lines in the fibers of \E and closed points of the projectivization
$\PP^*(\E)$ are hyperplanes in the fibers of~\E. In particular,
$p_*\Oo_{\PP^*(\E)|X}(1)=\E$, where $p\colon\PP^*(\E)\to X$ is the
canonical projection.

If $(\PP^N)^*$ is the dual projective space to $\PP^N$ and
$\alpha\in(\PP^N)^*$, then $H_\alpha\subset\PP^N$ is the corresponding
hyperplane. If $X\subset\PP^N$ is a smooth projective variety, then
its \emph{dual variety} is
\begin{equation*}
X^*=\{\alpha\in(\PP^N)^*\colon \text{$H_\alpha$ is not transversal to~$X$}\}.  
\end{equation*}
If $\dim X=n$, $\alpha\in(\PP^N)^*\setminus X^*$, $\dim X=n$, and
$H_\alpha \cap X=Y$, then the variation of $\alpha$ in
$(\PP^N)^*\setminus X^*$ induces an action of the fundamental group
$\pi_1((\PP^N)^*\setminus X^*,\alpha)$ on $H^{n-1}(Y,\Q)$ which is
called \emph{the monodromy action}; slightly abusing the language, we
will say that the image of $\pi_1((\PP^N)^*\setminus X^*,\alpha)$ in
$\GL(H^{n-1}(Y,\Q))$ is \emph{the monodromy group} of~$X$; as a
subgroup of $\GL_{b_{n-1}(Y)}(\Q)$, the monodromy group is defined
up to conjugation.

\section{A survey of some results from Picard--Lefschetz
  theory}\label{sec:PLsurvey} 

None of the assertions in this section claims to novelty: we just
briefly recall what we need from Picard--Lefschetz theory. Almost all
the proofs will be omitted. For non-trivial details we refer the
reader to SGA7.2, especially Expos\'es XVII and XIX (\cite{SGA7.2} in
the bibliography) and to Lamotke~\cite{Lamotke}.

Suppose that $X\subset\PP^N$ is a smooth projective variety of
\emph{odd} dimension~$n=2k+1$ and $H=H_\alpha\subset\PP^N$, where
$\alpha\in (\PP^N)^*\setminus X^*$, is a hyperplane transversal to
$X$; put $Y=H\cap X$.

\subsection{Pairings.}\label{subseq:pairings}
We will identify $H_0(Y,\Q)$ with \Q, mapping a singular $0$-chain
$\sum c_ia_i$, $a_i\in Y$, to $\sum c_i\in\Q$.  If $\xi\in H^m(Y,\Q)$,
$\sigma\in H_m(Y,\Q)$, we put, keeping this identification in mind,
\begin{equation}\label{eq:<>}
\langle\xi,\sigma\rangle=\xi\cap\sigma\in H_0(Y,\Q)=\Q.  
\end{equation}
The \Q-valued pairing $\langle\cdot,\cdot\rangle$ is non-degenerate.

Define, on $H^{2k}(Y,\Q)$, the
non-degenerate bilinear form
\begin{equation}\label{eq:PL:bil.form}
  (\xi,\eta)=(-1)^k\langle\xi\cup\eta,[Y]\rangle,
\end{equation}
where $[Y]\in H_{4k}(Y,\Q)$ is the fundamental class. If $\dim X=3$,
then the pairing~\eqref{eq:PL:bil.form} is the negated intersection index.

Let $P\colon H^{2k}(Y,\Q)\to H_{2k}(Y,\Q)$ be the Poincar\'e duality
isomorphism, defined by the formula $P(\xi)=\xi\cap[Y]$ (cap-product
with the fundamental class). The isomorphism $P$ transplants the
pairing~\eqref{eq:PL:bil.form} to $H_{2k}(Y,\Q)$, by the formula
\begin{equation}\label{eq:()}
(x,y)=(P^{-1}(x),P^{-1}(y)).  
\end{equation}
The notation coincides with that in~\eqref{eq:PL:bil.form}; this
will not lead to a confusion. For any $x,y\in H_{2k}(Y,\Q)$, one has
\begin{equation*}
(x,y)=\langle P^{-1}(x),y\rangle.  
\end{equation*}

\subsection{Spaces of vanishing cycles.}\label{Spaces-of-vanishing-cycles}
The mapping $i^*\colon H^{2k}(X,\Q)\to H^{2k}(Y,\Q)$, where $i\colon
Y\hookrightarrow X$ is the embedding, is injective, and the
pairing~\eqref{eq:PL:bil.form} is non-degenerate on
$i^*H^{2k}(X,\Q)$. Let $\Ev(Y)$ be the orthogonal complement to
$i^*H^{2k}(X,\Q)$ in $H^{2k}(Y,\Q)$ with respect to the bilinear
form~\eqref{eq:PL:bil.form}. One has
$\dim\Ev(Y)=b_{2k}(Y)-b_{2k}(X)$. If $X^*$ is a hypersurface,
$\Ev(Y)\ne0$ and the bilinear form~\eqref{eq:PL:bil.form} is
non-degenerate on~$\Ev(Y)$.  The variation of $H_\alpha$ in the family
of all transversal to $X$ hyperplanes induces an action of
$\pi_1(\PP^N\setminus X^*)$ on $H^{2k}(Y,\Q)$ preserving the bilinear
form and the decomposition $H^{2k}(Y,\Q)=i^*H^{2k}(X,\Q)\oplus
\Ev(Y)$; this action is trivial on $i^*H^{2k}(X,\Q)$ and irreducible
on $\Ev(Y)$.

Put
\begin{equation*}
\ev(Y)=\Ker(i_*\colon H_{2k}(Y,\Q)\to H_{2k}(X,\Q)).  
\end{equation*}
If $P\colon H^{2k}(Y,\Q)\to H_{2k}(Y,\Q)$ is the Poincar\'e duality
isomorphism, then $P(\Ev(Y))=\ev(Y)$.  The isomorphism
$P\colon\Ev(Y)\to\ev(Y)$ transplants the action of the
monodromy group on $\Ev(Y)$ to~$\ev(Y)$.

\subsection{Vanishing cycles and the monodromy
  group.}\label{subseq:vanishing} 
Suppose now that $\ell\subset (\PP^N)^*$ is a Lefschetz pencil with
respect to $X$ (see \cite[Expos\'e~XVII, 2.2]{SGA7.2}) and put
$\ell\cap X^*=\{\lst\alpha r\}$. If $\lst\delta r\in H_{n-1}(Y,\Z)$
are vanishing cycles corresponding to \lst\alpha r, there exists an
exact sequence
\begin{equation*}
\Z^r\xrightarrow A H_{n-1}(Y,\Z)\to H_{n-1}(X,\Z)\to 0,  
\end{equation*}
where $A(\lst nr)=\sum_{i=1}^r n_i\delta_i$ (see \cite[Expos\'e~XIX,
  Section~4.3]{SGA7.2}). For any $i$, $1\le i\le r$, put
\begin{equation*}
\tau_i=A(\delta_i)\otimes 1\in H_{2k}(Y,\Z)\otimes\Q=H_{2k}(Y,\Q).  
\end{equation*}
The elements $\tau_i$ are also called vanishing cycles.

The vanishing cycles $\tau_i$ have the following properties.

\begin{enumerate}
\item The subspace $\ev(Y)\subset H_{2k}(Y,\Q)$ is spanned by \lst\tau
  r.
\item For any $i,j$, $1\le i\le j\le r$, there exists an element $g$
  of the monodromy group such that $g(\tau_i)=\tau_j$.  
\item\label{(/tau,/tau)} $(\tau_i,\tau_i)=2$ for all $i$ (recall
  definitions~\eqref{eq:PL:bil.form} and~\eqref{eq:()}).
\item\label{gen.by.refl}
  The monodromy group acting on $\ev(Y)$ is generated by ``reflections
  in vanishing cycles'', that is, by the linear mappings of the
  form
  \begin{equation}\label{eq:refl}
s_i\colon x\mapsto x-\langle x,\tau_i^\vee\rangle \tau_i,\qquad 1\le i\le r,    
  \end{equation}
where $\lambda^\vee=P^{-1}(\lambda)\in\Ev(Y)$ for any
$\lambda\in\ev(Y)$.
\end{enumerate}

\subsection{Finite monodromy groups.}\label{survey:finite.monodromy}
Suppose, in the above setting, that $X^*$ is a hypersurface, so the
monodromy group $G\subset\GL(\Ev(Y))$ is nontrivial. According to
Proposition~3.4 in Expos\'e~XIX of~\cite{SGA7.2}, the group $G$ is
finite if and only if $\Ev(Y)\subset H^{2k}(Y,\Q)\cap H^{k,k}(Y)$.

Suppose now that this is the case and put
\begin{equation*}
R=\{g(\tau_i)\quad \text{for all $g\in G$ and all $i$, $1\le i\le r$}\}.  
\end{equation*}

\begin{proposition}[P.~Deligne]\label{Deligne}
The subset $R$ of the space $\ev(Y)$ endowed with the bilinear
form~\eqref{eq:()} is an irreducible root system of the
type~$A$, $D$, or~$E$, the root lattice of this root system is spanned
by \lst\tau r, and its Weyl group coincides with the monodromy
group~$G$.
\end{proposition}

This assertion, or rather an equivalent statement about $\Ev(Y)$ and
$l$-adic cohomology, is contained in~\cite[Expos\'e XIX,
  Proposition~3.3]{SGA7.2}. For the reader's convenience we provide
some details.

\begin{proof}[Proof of Proposition~\ref{Deligne}]
To show that $R$ is a root system we are to check the conditions
($\mathrm{SR_I}$)--($\mathrm{SR_{III}}$) from Bourbaki~\cite[Chapitre
  VI, \S\,1]{LIE}.  Observe that $R$ is finite since $G$ is finite and
that none of the elements of $R$ is zero since none of the~$\tau_i$ is
zero by virtue of assertion~(\ref{(/tau,/tau)}) from
Section~\ref{subseq:vanishing}. This checks
Condition~($\mathrm{SR_I}$). If $\tau=g(\tau_i)\in R$, where $g\in G$
and $1\le i\le r$, and if $s_i$ is the reflection
from~\eqref{eq:refl}, then the linear automorphism $g^{-1}s_ig\in G$
is the reflection in $\tau=g(\tau_i)\in R$, of the form $x\mapsto
x-(x,\tau^\vee)x$, so this reflection maps $R$ into itself and
Bourbaki's Condition~($\mathrm{SR_{II}}$) is also satisfied. Finally,
to check Condition~($\mathrm{SR_{III}}$) observe that all the $\tau_i$
lie in the image of the natural mapping $H_{2k}(Y,\Z)\to
H_{2k}(Y,\Q)$. Since the action of $\pi_1((\PP^N)^*\setminus X^*)$ on
$H_{2k}(Y,\Q)$ lifts to $H_{2k}(Y,\Z)$, this is true for all the
reflections in elements of $R$ as well. It is clear that if $x',y'\in
H_{2k}(Y,\Z)$ and $x$, $y$ are their images in $H_{2k}(Y,\Q)$, then
$\langle x,y^\vee\rangle=(-1)^kx'\cap y'\cap[Y]_\Z\in\Z$ (here,
$[Y]_\Z\in H_{4k}(Y,\Z)$ is the fundamental class and we naturally
identified $H_{4k}(Y,\Z)$ with~\Z). Thus,
Condition~($\mathrm{SR_{III}}$) is satisfied and $R$ is a root system.

The monodromy group $G$ coincides with the Weyl group of~$R$ by virtue
of assertion~\ref{gen.by.refl} from Section~\ref{subseq:vanishing},
the root system~$R$ is irreducible since the action of $G$ on $\ev(Y)$ is
irreducible, and $R$ is simply laced since the lengths of all roots are
equal by virtue of assertion~\ref{(/tau,/tau)} from
Section~\ref{subseq:vanishing}. It remains to show that the root
lattice is spanned by \lst\tau r, i.\,e., that any $\tau\in R$ is an
integer linear combination of \lst\tau r. To that end, it suffices to
show that each $s_i(\tau_j)$ is an integer linear combination of
\lst\tau r, and this assertion follows from
Condition~($\mathrm{SR_{III}}$), which we already checked.
\end{proof}

\begin{note}
Observe that \lst\tau r \emph{need not} be a base of the root
system~$R$.
\end{note}

Further on, the words, say, ``The variety $X$ has $A_2$ monodromy''
will mean that the monodromy group of $X$ is finite and the root
system~$R$ constructed above is of the indicated type; in particular,
it implies that dimension of the space of vanishing cycles equals the
rank of this root system and the monodromy group is isomorphic to its
Weyl group. We will refer to the root system~$R$ above as
the \emph{vanishing root system} of the variety~$X$.

\section{Sommese's list}\label{sec:case-by-case}

From now on, we assume that the odd-dimensional variety
$X\subset\PP^N$ with monodromy group $\Z/2\Z$ has dimension~$3$; we
aim at the classification of such varieties. We begin with varieties
having \emph{finite} monodromy group.

We will need the following \emph{ad hoc}
terminology, which is a variant of the notion of ``minimal
reduction'' from~\cite{Sommese}.

\begin{definition}\label{def:subordinate}
Let us say that a projective variety $Y\subset \PP^n$ is
\emph{$k$-subordinate} to the variety $X\subset\PP^{N+k}$ if there exists
a sequence of smooth projective varieties $X_i\subset \PP^{N+k-i}$,
$X_0=X$, $X_k=Y$, and birational morphisms
\begin{equation}\label{eq:subordinate}
  \xymatrix{
    {X=X_0}\ar@{-->}[r]^-{\ph_0}&{X_1}\ar@{-->}[r]^-{\ph_1}&{\ldots}
    \ar@{-->}[r]^-{\ph_{k-1}}&{X_k=Y}
  }   
\end{equation}
such that each $\ph_i\colon X_i\dasharrow X_{i+1}$ is the projection
from a point $a_i\in X_i$ and each $\ph_i$ induces an isomorphism
between $X_{i+1}$ and the blowup of $X_i$ at~$a_i$. If $Y$ is
$k$-subordinate to $X$ for some $k>0$, we will just say that $Y$ is
\emph{subordinate} to~$X$.
\end{definition}

\begin{proposition}\label{sommeses_list}
Suppose that $X\subset\PP^N$ is a smooth threefold such that $X^*$ is
a hypersurface \textup(equivalently, the
monodromy group of $X$ is non-trivial \textup). Then the monodromy
group of~$X$ is  finite if and only if $X$ is one of the varieties
listed in Table~\ref{table:Sommese's.list}.
\end{proposition}

\begin{table}
  \caption{Smooth threefolds $X\subset\PP^N$ with finite and
    nontrivial monodromy group}\label{table:Sommese's.list}

  \begin{tabular}{|c|p{.9\textwidth}|}
    \hline
No & Description of $X$\\
\hline
1 & $X$ is a scroll over a surface, that is, there exists a locally
  free sheaf~\E of rank~$2$ over a smooth surface~$S$ such that
  $(X,\Oo_X(1))\cong(\PP^*(\E),\Oo_{\PP^*(\E)|S}(1))$.\\
\hline
2 & $X$ is a pencil of quadrics, that is, there exists a morphism
$p\colon X\to C$, where $C$ is a smooth curve, such that the fiber of~$p$
  over a general point of $C$ is a smooth quadric (i.\,e., a smooth
  surface of degree~$2$ in~$\PP^N$).\\
\hline
3 & $X$ is a \emph{Veronese pencil}, that is, there exists a
  morphism $p\colon X\to C$, where $C$ is a smooth curve, such that,
for a general point $a\in C$, the fiber $X_a=p^{-1}(a)$ is a smooth
surface and $(X_a,\Oo_{X_a}(1))\cong (\PP^2,\Oo_{\PP^2}(2))$.\\
\hline
4 & $X$ is a Del Pezzo threefold, i.\,e., a Fano variety embedded by
  one half of the anticanonical class, that is,
  $\omega_X\cong\Oo_X(-2)$.\\
\hline
5 & $X$ is a smooth quadric in $\PP^4$.\\
\hline
6 & $X$ is the Veronese image $v_2(Q)\subset\PP^{13}$ or its isomorphic
  projection.\\
\hline
7 & $X\subset\PP^N$ is subordinate to $v_2(Q)$, that is, $X$ is the
blowup of the smooth three-dimensional 
  quadric $Q$ at $k\ge1$ points, and $\Oo_X(1)\cong
  \Oo_X(2\sigma^*H-E_1-\dots-E_k)$, where 
  $\sigma\colon X\to Q$ is the blowdown morphism, $H$ is a hyperplane
  section of~$Q$, and $\lst Ek\subset X$ are exceptional divisors.\\
\hline
  \end{tabular}
  
\end{table}

\begin{note}
The categories in Table~\ref{table:Sommese's.list} are not disjoint.
\end{note}

We will refer to the list in Table~\ref{table:Sommese's.list} as
\emph{Sommese's list}.

\begin{proof}
According to Proposition~6.1 of the paper~\cite{Lvovski}, the
monodromy group of $X$ is finite if and only if
$H^0(X,\omega_X(1))=0$. Now Main Theorem of Sommese's
paper~\cite{Sommese} contains a complete classification of pairs
$(X,\Ll)$, where $X$ is a smooth projective variety of dimension
$n\ge3$ and \Ll is an invertible sheaf on $X$ that is ample and
spanned by global sections and $H^0(X,\omega_X\otimes \Ll^{\otimes
  n-2}=0)$ (Sommese allows the variety~$X$ to have some mild
singularities as well). The list of such pairs is contained in
sections~(0.2), (0.3), and~(0.4) of~\cite{Sommese}.

Extracting smooth embedded threefolds with very ample~\Ll from the
list in Sommese's Section~(0.2) and excluding threefolds with trivial
monodromy (i.e., those whose dual is not a hypersurface, i.e.,
$\PP^2$-scrolls over curves), one obtains
smooth quadrics (item~(5) in Table~\ref{table:Sommese's.list}).

Applying the same procedure to the list in~\cite[Section
  (0.3)]{Sommese}, one obtains varieties from items~(1),
(2), and~(4).

Now, translating Sommese's result into the language of embedded
projective varieties, it is easy to see that the pairs $(X,\Oo_X(1))$ from
the list in~\cite[Section (0.4)]{Sommese}, where $\dim X=3$, $X$ is
smooth and embedded in $\PP^N$, are as follows. Either
$(X,\Oo_X(1))\cong (\PP^3,\Oo_{\PP^3}(2))$, or $(X,\Oo_X(1))\cong
(Q,\Oo_Q(2))$, where $Q\subset\PP^4$ is a smooth quadric, or there
exists a morphism $p\colon X\to C$, where $C$ is a curve, such that
$\omega_X^{\otimes 2}(3)\cong p^*\Ll$, where \Ll is a line bundle
on~$C$, or, finally, $X$ is subordinate to one of the above. 

If $(X,\Oo_X(1))\cong (\PP^3,\Oo_{\PP^3}(2))$, then $X$ is a variety
in item (4) of Sommese's list (Table~\ref{table:Sommese's.list}). If
one projects such an~$X$ from a point $a\in X$ to obtain a smooth
variety $X'\subset\PP^{N-1}$ that is isomorphic to the blowup of $X$
at $a$, then $X'$ is also a variety in item (4), and it is impossible
to further extend the chain from Definition~\eqref{def:subordinate}
since the variety~$X'$ is swept by lines, which will be blown down by
the next projection.

If $(X,\Oo_X(1))\cong
(Q,\Oo_Q(2))$, where $Q$ is the three-dimensional quadric, then $X$ is
$v_2(Q)\subset\PP^{13}$ or its isomorphic projection; this is
item~(6) in Sommese's list.

If $X$ is subordinate to a variety from item~(6), then it is a variety
from item~(7) (observe that, in the chain of
projections~\eqref{eq:subordinate}, each point $a_i$ does not lie on
exceptional divisors of the previous blowups since otherwise the
projection would blow down the lines lying on this exceptional divisor
and passing through~$a_i$).

Finally, if there exists a morphism $p\colon X\to C$ such that
$\omega_X^{\otimes 2}(3)\cong p^*\Ll$, where $\Ll$ is an invertible sheaf
on $C$, then, denoting by $F\subset X$ a general fiber of $p$ and
restricting to $F$, one obtains the isomorphism
$\omega_F^{\otimes2}\cong \Oo_F(-3)$, whence $F\cong \PP^2$ (since $F$
is a Del Pezzo surface and the canonical class of~$F$ is divisible
by~$3$) and $\Oo_F(1)\cong\Oo_{\PP^2}(2)$; thus, such an $X$ belongs
to item~(3) in Table~\ref{table:Sommese's.list}, so $X$ is a
Veronese pencil. Since a variety subordinate to a
Veronese pencil is also a Veronese pencil, this completes the proof.
\end{proof}

\section{Inspection of Sommese's list}

In this section we will find vanishing root systems for varieties from
Sommese's list. The method for identification of root systems was used
by Yu.~I.~Manin~\cite[Chapter 4]{Manin}, who attributes it to Deligne.

We will use the following  
\begin{Not}
Suppose that $L$ is a lattice in a finite dimensional linear space $V$
over~\Q, and suppose that $V$ is endowed with a symmetric and
non-degenerate definite \Q-valued bilinear form $(\cdot,\cdot)$ such
that $(x,y)\in\Z$ for any $x,y\in L$. We will put
\begin{equation*}
L^\vee=\{\lambda\in L\otimes\Q\colon(x,\lambda)\in\Z\quad \text{for all
  $x\in L$}\}.
\end{equation*}
\end{Not}

\begin{note}\label{newnote}
If $\langle \lst er\rangle$ is a \Z-basis of~$L$, then index of $L$ in
$L^\vee$ equals $\bigl|\det\|(e_i,e_j)\|\bigr|$. Hence, if $L\subset
L_1$ are two lattices such that $(x,y)\in\Z$ for any $x,y\in L_1$,
then
\begin{equation*}
[L^\vee:L]=[L_1^\vee:L_1]\cdot [L_1:L]^2.
\end{equation*}
\end{note}

\subsection{Scrolls over a surface}\label{ntl:subseq:scrolls}
We begin with item~(1) of Sommese's list
(Table~\ref{table:Sommese's.list}).

Throughout this section, $X$ is a three-dimensional scroll over a
surface, that is, $X=\PP^*(\E)$, where $\E$ is a very ample
bundle of rank~$2$ over a smooth projective surface $S$ and 
$\Oo_X(1)=\Oo_{X|S}(1)$. 

The answer to the question which scrolls over surfaces have $A_1$
monodromy can be immediately read off from the book~\cite{BelSom}.

\begin{proposition}\label{c_2=2}
In the above setting, $X$ has $A_1$ monodromy if and only if
$(X,\E)\cong (\PP^2,\Oo_{\PP^2}(1)\oplus \Oo_{\PP^2}(2))$ or
$(X,\E)\cong (Q,\Oo_Q(1)\oplus \Oo_Q(1))$, where $Q\subset\PP^3$ is
the smooth quadric.  
\end{proposition}

\begin{proof}
Observe that for any smooth threefold $X$ with a smooth hyperplane
section~$Y$, the assertions ``$X$ has $A_1$ monodromy'' and
``$b_2(Y)=b_2(X)+1$'' are equivalent. Indeed, the only reflection on a
one-dimensional linear space is $x\mapsto-x$, so if $b_2(Y)=b_2(X)+1$,
then the monodromy group that acts on the one-dimensional $\Ev(Y)$,
being generated by reflections, is isomorphic to $\Z/2\Z$,
whence~$A_1$; the opposite implication is trivial.
  
Now if $Y\subset X$ is a general smooth hyperplane section, where $X$
is a $\PP^1$-scroll over a surface~$S$, then the hyperplane
section~$Y$ is isomorphic to $S$ with $c_2(\E)$ points blown up, so
$b_2(Y)=b_2(S)+c_2(\E)$. Since $b_2(X)=b_2(S)+1$, the variety $X$ has
$A_1$ monodromy (equivalently, $b_2(Y)=b_2(X)+1$) if and only if
$c_2(\E)=2$, and Theorem~11.4.5 from~\cite{BelSom} does the job.
\end{proof}

\begin{note}
One observes that if $(X,\E)\cong (Q,\Oo_Q(1)\oplus \Oo_Q(1))$ then $X$
is the Segre variety $\PP^1\times\PP^1\times\PP^1\subset\PP^7$
(variety~$V_6$ in Table~\ref{table:Fano2} below), and if $(X,\E)\cong
(\PP^2,\Oo_{\PP^2}(1)\oplus \Oo_{\PP^2}(2))$ (variety $V_7$ in
Table~\ref{table:Fano2}) then $X$ is isomorphic to the projection of
the Veronese variety $v_2(\PP^3)\subset\PP^9$ from a point lying on
$v_2(\PP^3)$, or to an isomorphic projection of this projection.
\end{note}

For the sake of completeness we find the vanishing root system for
other scrolls over surfaces as well.

\begin{proposition}\label{NTL:P^1-bundles}
In the above setting,
if $c_2(\E)=1$, then $X^*$ is not a hypersurface in $(\PP^N)^*$, and if
$c_2(\E)=r>1$, then the vanishing root system of~$X$ is~$A_{r-1}$.
\end{proposition}

\begin{proof}
Let $p\colon X\to S$ be the natural projection and let $Y\subset X$ be a
general smooth hyperplane section; $Y$ is isomorphic to $S$ with
$r$ points blown up. Let $\lst \ell r\subset Y\subset X$ be the
corresponding exceptional curves, which are lines on~$Y$ and fibers of
the projection $p$ on~$X$. If $l_j\in H_2(Y,\Z)$ is the class of
$\ell_j$, then the images of all the $l_j$ in $H_2(X,\Z)$ are equal,
whence the image of $\Ker(H_2(Y,\Z)\to H_2(X,\Z))$ in $H_2(Y,\Q)$,
with the pairing $(\cdot,\cdot)$ from Section~\ref{subseq:pairings}
(i.\,e., with the negated intersection index), is isomorphic to the
lattice
\begin{equation*}
  L=\left\{\sum_{j=1}^rc_jl_j\colon c_j\in\Z,\
  \sum_{j=1}^r c_j=0\right\},
\end{equation*}
where $(l_i,l_j)=\delta_{ij}$. If $r=1$, this implies that $\ev(Y)=0$,
so $X^*$ is not a hypersurface. If $r>1$, then $\rank L=r-1$, and a
well-known computation (see for example \cite[Planche~I, (VIII)]{LIE})
shows that $[L^\vee:L]=r$. For a simply laced irreducible root system
of rank~$r-1$, this is possible only if it is~$A_{r-1}$.
\end{proof}

We end Section~\ref{ntl:subseq:scrolls} with two by-products of
Proposition~\ref{NTL:P^1-bundles}.  The first one concerns the
monodromy action on zero loci of sections of vector bundles on
surfaces.
 
Suppose that $\E$ is a very ample bundle of rank~$2$ on a smooth
surface~$S$, $c_2(\E)=r$. Put $V=H^0(S,\E)$ and let $D\subset V$ be the
set of sections that are not transversal to the zero section. Any
section in $V\setminus D$ has precisely $r$ zeroes. Fix a section
$s_0\in V\setminus D$, and let $Z(s_0)=\{\lst pr\}$ be its zero
locus. As sections of \E vary in $V\setminus D$, the fundamental group
$\pi_1(V\setminus D,s_0)$ acts on $Z(s_0)$ by permutations. Let us say
that the image of $\pi_1(V\setminus D,s_0)$ in the group of
permutations of \lst pr is \emph{the monodromy group of sections
  of~$\E$}. 

\begin{proposition}\label{prop:zeroes-of-sections}
In the above setting, the monodromy group of sections of~$\E$ is the
entire group of permutations of the set $Z(s_0)$.
\end{proposition}

\begin{proof}
Denoting the group in question by~$\Gamma$, put $X=\PP^*(\E)$ and
embed $X$ into $\PP^N$ via the complete linear system
$|\Oo_{X|S}(1)|$. Let $Y\subset X$ be the smooth hyperplane section
of~$X$ corresponding to the section $s_0\in H^0(S,\E)\cong
H^0(X,\Oo_{X|S}(1))$; the surface $Y$ is isomorphic to the blowup of
$S$ at the points of $Z(s_0)$, so the monodromy transformations acting
on $H_2(Y,\Q)$ are induced by the permutations from $\Gamma$. Thus,
$\Gamma$ surjects onto the monodromy group
of~$X$. Proposition~\ref{NTL:P^1-bundles} implies that the monodromy
group of $X$ is the symmetric group~$S_r$, so $\Gamma$ surjects
onto~$S_r$. Since, on the other hand, $\Gamma$ is a subgroup of the
permutation group $\Aut(Z(s_0))\cong S_r$, we conclude that
$\Gamma=\Aut(Z(s_0))$.
\end{proof}

Another by-product of Proposition~\ref{NTL:P^1-bundles} is an
assertion to the effect that if
$X\subset\PP^n$ is a scroll over a surface then the curve that is its
general linear section of codimension~$2$ cannot be ``too special''.

If $\pi\colon C\to C_1$ is a finite morphism of (smooth, projective,
and connected) algebraic curves and if $x_0\in C_1\setminus B$, where
$B\subset C_1$ is the branch locus of $\pi$, then $\pi_1(C_1\setminus
B)$ acts on $\pi^{-1}(x_0)$; the induced subgroup in the group of
permutations of $\pi^{-1}(x_0)$ will be called \emph{monodromy group
  of the fibers of~$\pi$}.

\begin{proposition}\label{linear-section}
Suppose that $X=\PP^*(\E)\subset\PP^N$, where $\E$ is a rank $2$ very
ample bundle on a smooth surface~$S$, is embedded with
$|\Oo_{X|S}(1)|$. If $c_1(\E)=r$ and $C$ is a general one-dimensional
linear section of $X$, then there exists a morphism $\pi\colon
C\to\PP^1$ such that $\deg\pi=r$ and the monodromy group of the fibers
of~$\pi$ is the entire symmetric group~$S_r$.  
\end{proposition}

\begin{note}
The assertion that $C$ possesses a morphism of degree $r$ onto
$\PP^1$ is a particular case of Theorem~11.1.2(2)
from~\cite{BelSom}. It is only the second part of this proposition
that claims to novelty.  
\end{note}

\begin{proof}
By lines of the ruling we will mean the lines on $X$
that are fibers of the natural projection~$X\to S$.

Suppose that $\ell\subset(\PP^N)^*$ is a general Lefschetz pencil and
$L\subset \PP^N$, where $\dim L=N-2$, is its axis. For a general
$\ell$, the linear subspace~$L$ does not contain any line of the
ruling of~$X$, so we may and will assume that this is the
case. Putting $C=X\cap L$, define the mapping $\pi\colon C\to \ell$ as
follows. For any $x\in C$, put $\pi(x)=\alpha\in \ell$, where the
hyperplane $H_\alpha\subset \PP^N$ is the linear span of~$L$ and the
line of the ruling passing through~$x$. (This is just a geometric
description of the morphism~$C'\to\PP^1$ from the proof
of~\cite[Theorem~11.1.2(2)]{BelSom}.) If $\alpha\in\ell$ is such that
$H_\alpha$ is transversal to $X$, then $H_\alpha\cap X$ is isomorphic to the
blowup of $S$ at the points \lst sr that are the zeroes of the section
$\sigma\in H^0(S,\E)=H^0(X,\Oo_X(1))$ defining the
hyperplane~$H_\alpha$, and $H_\alpha$ contains precisely $r$ lines of
the ruling that are fibers over \lst sr.
If $\alpha\in \ell$ is such that $H_\alpha\cap
X$ is singular, then $H_\alpha$ contains precisely $r-1$ lines of the
ruling, on one of which the unique singular point of $H_\alpha\cap X$
sits. 

If we fix an $\alpha_0\in\ell$ for which $H_\alpha\cap X$ is smooth,
then it is clear from the proof of Proposition~\ref{NTL:P^1-bundles}
that the action of the monodromy group on $\ev(X\cap H_\alpha)$ is
induced by the monodromy permutations of lines of the ruling lying on
$X\cap H_\alpha$, i.\,e., of the points of fiber $\pi^{-1}(\alpha)$. Thus, the
latter group of permutations must be the entire~$S_r$.
\end{proof}

\subsection{Pencils of quadrics}\label{NTL:generalities}

Throughout this section $X\subset\PP^N$ will be a pencil of quadrics
in the sense of Proposition~\ref{sommeses_list} (item~(2) in
Table~\ref{table:Sommese's.list}).  That is, $X\subset\PP^N$ is a
smooth projective threefold and $p\colon X\to C$ is a morphism onto a
smooth curve~$C$ such that its fibers are isomorphic to quadrics
in~$\PP^3$. For each $t\in C$, put $X_t=p^{-1}(t)$.

It is well known that the subset $S\subset C$ of points such that
fibers over them are not smooth is finite and fibers over the points of~$S$
are quadratic cones of corank one.

We begin with a simple observation.

\begin{lemma}\label{POQ=>ev>0}
If $X\subset\PP^N$ is a pencil of quadrics and $Y\subset X$ is a
smooth hyperplane section, then $\dim\ev(Y)>0$.   
\end{lemma}

\begin{proof}
If $\ev(Y)=0$, then the dual variety $X^*$ is not a hypersurface, so
$X$ is swept by planes (see \cite[Theorem 5.1]{Ein2}). Since any
morphism from $\PP^2$ to a curve is constant, these planes should be
contained in the fibers of the morphism~$p\colon X\to C$, which is
absurd.
\end{proof}

If $t_0\in C\setminus S$, then $\pi_1(C\setminus S,t_0)$ acts by
monodromy on $H_2(X_{t_0},\Q)$. Since this action preserves the
intersection index and the class of the hyperplane section of the
smooth quadric~$X_{t_0}$, the monodromy transformation is either
identity, or it swaps the classes of the lines $\ell,m\subset X_{t_0}$
of different rulings.

We will distinguish between two types of pencils of quadrics.

\begin{definition}
Let us say that a pencil of quadrics $p\colon X\to C$ is
\emph{ordinary} if this monodromy action on $H_2(X_t,\Q)$ is
nontrivial, and that it is \emph{extraordinary} if this action is
trivial.
\end{definition}

\begin{note}\label{NTL:sing=>swap}
If $S\ne\varnothing$, then the pencil is ordinary. Indeed, if $t$
travels along a small circle around a point in~$S$, then the classes
of the lines of two rulings $\ell,m\subset X_t$ are interchanged.
\end{note}

Let $Y\subset X$ be a smooth hyperplane section, and let $\pi\colon
Y\to C$ be the restriction of $p$ to~$Y$. It is clear that singular
fibers of $\pi$ are pairs of different intersecting lines.  Let
$T=\{\lst tm\}\subset C$ be the set of points such that fibers
of~$\pi$ over them are singular. Observe that $T\ne\varnothing$: a
simple dimension count shows that no hyperplane in $\PP^N$ can be
transversal to all the fibers of $p\colon X\to C$.

For each $t_j\in T$, choose once and for all a line $\ell_j\subset
\pi^{-1}(t_j)$. Each $\ell_j$ is a $(-1)$-curve on $Y$; blowing down
all the $\ell_j$, one obtains a smooth surface $Y'$ which is a
$\PP^1$-bundle over $C$. Let $\sigma\colon Y\to Y'$ be the blowdown
morphism. A standard argument using the fact that the field of
rational functions on a curve is a $C_1$ field shows that the
projection $\pi'\colon Y'\to C$ has a section, so there exists a
divisor $D\subset Y'$ having intersection index~$1$ with all the
fibers of the projection $\pi'$. Let $e\in H_2(Y,\Z)$ be the class of
the divisor $\sigma^*D$, and let $f\in H_2(Y,\Z)$ be the class of the
divisor that is $\sigma^*$ of a fiber of the projection $\pi'$. It is
clear that $H_2(Y,\Z)$ is a free abelian group with basis $\langle
e,f, \lst lm\rangle$, where $l_j$, $1\le j\le m$, are the classes of
$\ell_j$. The class $e$ has intersection index one with~$f$ and zero
with each~$l_j$.

\begin{proposition}\label{NTL:ord-quadr-pencils}
Suppose that $X\subset \PP^N$ is an ordinary pencil of quadrics over a
curve~$C$ and $Y\subset X$ is a smooth hyperplane section; let $m$ be
the number of degenerate fibers of the induced morphism $\pi\colon
Y\to C$.  Then $\dim\ev(Y)=m$; the vanishing root system of~$X$ is
$D_m$ if $m\ge4$, it is $A_3$ if $m=3$, and the cases $m=1,2$ are
impossible.
\end{proposition}

\begin{proof}
Let $i\colon Y\hookrightarrow X$ be the embedding and let $i_*\colon
H_2(Y,\Z)\to H_2(X,\Z)$ be the induced homomorphism. Suppose that
\begin{equation}\label{ntl:quadrics:relation}
ae +bf+\sum_{j=1}^m c_jl_j\in\Ker i_*.   
\end{equation}
Intersecting $i_*$ of the left-hand side
of~\eqref{ntl:quadrics:relation} with the fiber of~$p$, one obtains
$a=0$. Since the pencil $X$ is ordinary, all the $i_*(l_j)$ are
equal to the same element $l\in H_2(X,\Z)$, and $i_*(f)=2l$. Thus,
\eqref{ntl:quadrics:relation} is equivalent to
\begin{equation}\label{ntl:quadrics:relation2}
  c_1+\dots+c_m+2b=0,
\end{equation}
so the root lattice of the vanishing root system of $X$ is
\begin{equation*}
L=\{bf+\sum_{j=1}^m c_jl_j\colon c_1+\ldots+c_m=-2b\}. 
\end{equation*}
If $(\cdot,\cdot)$ is the pairing on $H_2(Y,\Q)$ defined as
in~\eqref{eq:()}, Section~\ref{subseq:pairings}, then in $H_2(Y,\Q)$
one has $(f,l_j)=(f,f)=0$ for all $j$ and
$(l_i,l_j)=\delta_{ij}$. Putting $\tilde l_i=l_i-\frac12f$, one has
$(\tilde l_i,\tilde l_j)=\delta_{ij}$ and
\begin{equation}\label{eq:D_m}
  L=\left\{\sum_{j=1}^m c_j\tilde l_j\colon c_j\in\Z,\ \sum_{j=1}^m c_j\equiv
  0\pmod 2\right\},\quad\text{where $(\tilde l_i,\tilde l_j)=\delta_{ij}$.}
\end{equation}
It follows from Lemma~\ref{POQ=>ev>0} that $m>0$.

A well-known computation (see for example
\cite[Planche IV, (VIII)]{LIE}) shows that the order of the group
$L^\vee/L$ equals~$4$. So, the
vanishing root system is a simply laced irreducible root system of
rank~$m$ and the index of the root lattice in the weight lattice
is~$4$. For $m\ge 4$, such a system is $D_m$; for $m=3$, such a system
is~$A_3$; for $m=1$ or~$2$, such a system does not exist.
\end{proof}

\begin{proposition}\label{NTL:extraord.quadr}
Suppose that $X\subset\PP^N$ is an extraordinary pencil of quadrics
over a curve~$C$ and $Y\subset X$ is a smooth hyperplane section; let
$m$ be the number of degenerate fibers of the induced morphism $Y\to
C$. Then $\dim\ev(Y)=m-1$ and the vanishing root system of $X$ is
$A_{m-1}$.
\end{proposition}

\begin{proof}
Since the pencil $X$ is extraordinary, all its fibers are smooth (see
Remark~\ref{NTL:sing=>swap}), and if $Q$ is the generic fiber (in the
scheme-theoretic sense) of the morphism~$p$, then both families of
lines on the quadric $Q$ are defined over $\C(C)$. 

Hence, there exist two $2$-dimensional families of lines \U and~\V on~$X$
such that, for each~$t$, the family of lines from~\U (resp.~\V) lying
on $X_t$ is the family of all lines of one of the rulings of the
quadric $X_t$.

Since the base of the family of the lines of each ruling of a smooth
two-dimensional quadric is isomorphic to $\PP^1$, the same argument by
which we showed that the morphism $\pi'\colon Y'\to C$ has a section
shows that there exists a
surface $F\subset X$ such that, for each $t\in C$, $F\cap p^{-1}(t)$
is a line from the family \U. Now it is clear that if a line
$\ell\subset X$ belongs to the family \U, then the intersection index
$(\ell,F)$ is equal to zero, and if $\ell$ belongs to the family \V,
then $(\ell,F)=1$.  Hence, if $u\in H_2(X,\Z)$ (resp.~$v\in
H_2(X,\Z)$) stands for the class
of any line from the family~\U (resp.~\V), then the elements $u$ and
$v$ are linearly independent in $H_2(X,\Z)$.

Using notation from the beginning of this section
we may, without loss of generality, assume that all the lines $\ell_j$
we have chosen on the surface $Y$ belong to the family \U.  Suppose
now that \eqref{ntl:quadrics:relation} holds, where $i\colon Y
\hookrightarrow X$ is the embedding. Take $i_*$ of the left-hand side;
intersecting the resulting homology class with the class of a fiber
of~$X_t$ for some $t\in C$, one obtains $a=0$. Since, in $H_2(X,\Z)$,
$i_*(f)=u+v$, one infers that
\begin{equation*}
  b(u+v)+(c_1+\dots+c_m)u=0;
\end{equation*}
since $u$ and~$v$ are linearly independent in $H_2(X,\Z)$, this
implies that $b=0$ and $c_1+\dots+c_m=0$. Thus, the root lattice of
the vanishing root system is isomorphic to the lattice
\begin{equation*}
  L=\left\{\sum_{j=1}^mc_jl_j\colon c_j\in\Z,\ c_1+\dots+c_m=0\right\},
  \quad\text{where $(l_i,l_j)=\delta_{ij}$}
\end{equation*}
(Lemma~\ref{POQ=>ev>0} implies that $m>1$).  In the proof of
Proposition~\ref{NTL:P^1-bundles} we have seen that this implies that
the root system in question is~$A_{m-1}$.
\end{proof}

Now we can find out which pencils of quadrics have vanishing root
system~$A_1$.

\begin{proposition}\label{PoQ:A_1}
Suppose that $p\colon X\to C$, where $X\subset\PP^N$, is a pencil of
quadrics. Then its vanishing root system is $A_1$ if and only if
$X=\PP^1\times\PP^1\times\PP^1\subset\PP^7$ \textup(the Segre
embedding\textup) and $p$ is the projection onto the first factor.
\end{proposition}

\begin{proof}
Proposition~\ref{NTL:ord-quadr-pencils} shows that such
a pencil cannot be ordinary, so we can and will assume that $p\colon
X\to C$ is an extraordinary pencil.
Put $\E=R^0p_*\Oo_X(1)$. It is clear that $\E$ is a locally free sheaf
of rank~$4$ on~$C$ and that there exists an invertible sheaf \Ll on $C$
and a section $s\in H^0(\Sym^2(\E)\otimes\Ll)$ such that $X$ is the zero
locus of~$s$ in $\PP^*(\E)$.

Since the pencil is extraordinary, Remark~\ref{NTL:sing=>swap} implies
that it has no degenerate fibers. Thus, the discriminant
of $s$ as a family of quadratic forms, which is a section of
$(\det \E)^{\otimes2}\otimes\Ll^{\otimes4}$, has no zeroes
on~$C$. Hence,
\begin{equation}\label{NTL:eq:5,10}
  2\deg \E+4\deg\Ll=0.
\end{equation}
Now suppose that $Y\subset X$ is a transversal hyperplane section. If
$Y$ is the zero locus of a section $\sigma\in H^0(\Oo_X(1))\cong H^0(\E)$,
then, if $\E'$ is the quotient in the exact sequence
\begin{equation*}
  0\to\Oo_C\xrightarrow\sigma \E\to \E'\to 0,
\end{equation*}
one sees that $Y\subset\PP^*(\E')$ is the zero locus of a section
$\sigma'\in H^0(\Sym^2(\E')\otimes\Ll)$ and that the discriminant of
$\sigma'$ is a section of $(\det \E')^{\otimes2}\otimes\Ll^{\otimes3}$;
since the vanishing root system is $A_1$,
Proposition~\ref{NTL:extraord.quadr} implies that there are precisely
$2$ degenerate fibers of the induced pencil $Y\to C$, so the degree of
the invertible sheaf~$(\det \E')^{\otimes2}\otimes\Ll^{\otimes3}$
equals~$2$, whence
\begin{equation}\label{NTL:eq:4,6}
  2\deg \E'+3\deg\Ll=2.
\end{equation}
Taking into account that $\deg \E'=\deg \E$ and putting together
equations~\eqref{NTL:eq:4,6} and~\eqref{NTL:eq:5,10}, one sees that
$\deg \E=4$ and $\deg \Ll=-2$.

Put $Z=\bigcup_{t\in C}\langle X_t\rangle\subset\PP^N$. The natural
homomorphism $H^0(\Oo_{\PP^N}(1))\to H^0(\Oo_{X|C}(1))=H^0(\E)$ induces
a morphism $\ph\colon \PP^*(\E)\to \PP^N$ such that
$\ph^*\Oo_Z(1)=\Oo_{\PP^*(\E)|C}(1)$ and $\ph(\PP^*(\E))=Z$. Since $Z$
contains a one-dimensional family of $\PP^3$'s, it is not a quadric;
since $\deg(\ph|_{\PP^*(\E)})\cdot\deg Z=\deg\det \E=4$, the morphism
$\ph|_{\PP^*(\E)}$ is birational onto $Z$ and $\deg Z=4$. According to the
classification of varieties of degree~$4$ from the
paper~\cite{SD_deg4}, the only such varieties containing a
one-dimensional family of $\PP^3$'s are the Segre variety
$\PP^1\times\PP^3\subset\PP^7$ and its regular projections. Thus,
$C\cong\PP^1$ and $\E\cong\bigoplus_{i=1}^4\Oo(d_i)$, where $0\le
d_1\le d_2\le d_3\le d_4$ and $d_1+d_2+d_3+d_4=4$. Since $C\cong\PP^1$
and $\deg\Ll=-2$, one has $\Ll\cong\Oo_{\PP^1}(-2)$.

The quadratic form defining $X\subset\PP^*(\E)$, which is a section of
$\Sym^2(\E)\otimes\Ll$, can be represented as a
$4\times4$-matrix~$\|a_{ij}\|_{1\le i,j\le4}$, where $a_{ij}\in
H^0(\Oo_{\PP^1}(d_i+d_j-2))$. Let us show that the case $d_1=0$ is
impossible. Indeed, if this is the case, then $a_{11}$ is identically
zero, so in each fiber of the bundle $\PP^*(\E)$ the point with
homogeneous coordinates $(1:0:0:0)$ lies in~$X$ (we use homogeneous
coordinates that agree with the decomposition
$E=\bigoplus\Oo_{\PP^1}(d_i)$). On the other hand, since $d_1=0$, the
mapping $\ph\colon \PP(\E)\to\PP^N$ maps the points with coordinates
$(1:0:0:0)$ in all the fibers of $\PP^*(\E)$ to one and the same point
of $\PP^N$. Thus, there exists a point contained in all the fibers of
the pencil $p\colon X\to\PP^1$, which is absurd.

We have proved that $d_1\ne0$, whence $d_1=d_2=d_3=d_4=1$. Thus, $Z$
is the Segre variety $\PP^1\times\PP^3\subset\PP^7$, and the matrix
$\|a_{ij}\|$ consists of constants. This proves that $X=\PP^1\times
Q\subset\PP^1\times\PP^3$, where $Q\subset\PP^3$ is a smooth
quadric, and the proposition follows.
\end{proof}

\subsection{Veronese pencils}
In this section we account for item~(3) in Sommese's list
(Table~\ref{table:Sommese's.list}). We will see that none of such
varieties has monodromy group of the type~$A_1$.

Throughout this section $X\subset\PP^N$ will be a smooth projective
threefold such that there exists a morphism $p\colon X\to C$ onto a
smooth curve~$C$; for $t\in C$, we put $X_t=p^{-1}(t)$. We assume
that for a general $t\in C$ one has $(X_t,\Oo_{X_t}(1))\cong
(\PP^2,\Oo_{\PP^2}(2))$.

Observe that if, in this setting, $Y\subset X$ is a smooth hyperplane
section, then $\dim\ev(Y)\ne0$; the proof is the same as that of
Lemma~\ref{POQ=>ev>0}.

Let $S=\{\lst ur\}\subset C$ be the set of points such that fibers of
$p$ over them are not irreducible and reduced surfaces (of course, $S$
may be empty); each fiber~$X_{u_j}$ is the union of $n_j$ irreducible
components.  We begin with simple lemmas.

\begin{lemma}\label{v3:union}
Suppose that $\lst Fn\subset\PP^N$ are irreducible projective
surfaces satisfying the following conditions:

\textup{(1)} for any two different $F_i$, $F_j$, 
the set-theoretic intersection $F_i\cap F_j$ is
isomorphic to~$\PP^1$ or empty;

\textup{(2)} the graph of which the vertices are \lst Fn and such that
the vertices $F_i$ and $F_j$ are joined by an edge if and only if
$F_i\cap F_j\ne\varnothing$ \textup(i.e., the incidence graph of $\lst
Fn$\textup), is a tree.

\textup{(3)} $H^3(F_j,\Q)=0$ for each~$j$.

Then $H^3(F_1\cup\dots\cup F_n,\Q)=0$.
\end{lemma}

\begin{proof}
Condition (2) implies that the surfaces \lst Fn may be ordered so
that, for each~$m$, $F_m\cap F_{m-1}\ne\varnothing$ and $F_m\cap
F_j=\varnothing$ for $j<m$. Now we prove, by induction on~$m$, that
$H^3(F_1\cup\dots\cup F_m,\Q)=0$ for all~$m$.  In the induction step
one uses the Mayer--Vietoris sequence for $(F_1\cup\dots\cup
F_{m-1})\cup F_m$, taking into account that $C=(F_1\cup\dots\cup
F_{m-1})\cap F_m$ is isomorphic to $\PP^1$ and the mapping
\begin{equation*}
H^2(F_1\cup\dots\cup F_{m-1},\Q)\oplus H^2(F_m,\Q)  \to H^2(C,\Q)
\end{equation*}
is an epimorphism.
\end{proof}

\begin{lemma}\label{two-planes}
Suppose that $X\subset\PP^N$ is a smooth three-dimensional projective
variety which is not a linear subspace of $\PP^N$.  

\textup{(a)} If $X$ contains two planes $L_1$ and $L_2$, then $L_1\cap
L_2=\varnothing$.

\textup{(b)} If $X$ contains a plane $L$ and an irreducible
quadric~$Q$, then either $L\cap Q=\varnothing$ or $L\cap Q$ is
\textup(set-theoretically\textup) a line.
\end{lemma}

\begin{proof}
The proofs of both assertions being similar, we will
prove~(b). If, by way of contradiction, $C=L\cap Q$ is not empty and
not a line, then $C$ is an irreducible conic or
the union of two distnct lines. Hence, if $H$ is the linear span of $Q\cup
L$, then $\dim H=3$ and $H= T_x(Q\cup L)$ for any $x\in Q\cap L$. Since $T_x(Q\cup L)\subset
T_xX$ and $\dim T_xX=3$, we conclude that $T_xX=H$ for all $x\in C$,
so the Gauss map of $X$ is not finite. This contradicts Zak's theorem on tangencies~\cite[Chapter~I,
  Corollary~2.8]{zakbook}.
\end{proof}

\begin{lemma}\label{v3:H^3=0}
For each $u\in S$ one has $H^3(X_u,\Q)=0$.  
\end{lemma}

\begin{proof}
Since $C$ is a smooth curve, both $X_u$ and its
hyperplane section are connected. Since $\dim X=3$, the intersection of
any two components of $X_u$ is purely one-dimensional. Now we
consider two cases.

(a) 
The divisor $p^*u$ has no
multiple components.

If $H$ is a general hyperplane (which does not contain $X_u$), then
the scheme $H\cap X_u$ is reduced and connected, whence
$h^0(\Oo_{X_u})=1$.  Since the Hilbert polynomial of $X_u$ is the same
as that of $v_2(\PP^2)$, the arithmetic genus 
of $H\cap X_u$ equals zero.

It follows from the above that all the components of $H\cap X_u$ are
isomorphic to $\PP^1$ and the incidence graph of the components of
$H\cap X_u$ is a tree; in particular, any two components of $H\cap
X_u$ are either disjoint or intersect at precisely one point. Hence,
if $F_1$ and $F_2$ are two components of $X_u$, then either $F_1\cap
F_2=\varnothing$ or $F_1\cap F_2$ is a line. Besides, the incidence
graph of the components of $X_u$ is a tree and any component of $X_u$
is a surface of degree~$\le 3$ whose general hyperplane sections are
isomorphic to $\PP^1$. It follows from the latter assertion that
$H^3(F,\Q)=0$ whenever $F$ is a component of $X_u$. Thus, $X_u$
satisfies the hypothesis of Lemma~\ref{v3:union}, whence
$H^3(p^{-1}(u),\Q)=0$.

(b) The divisor $p^*u$ has a multiple component. Since $\deg p^*u=4$
and the set~$X_u=p^{-1}(u)$ is connected, Lemma~\ref{two-planes} implies that
$X_u$ is either a plane, or an irreducible quadric, or the union
of an irreducible quadric and a plane intersecting in a line. In each of these
cases the vanishing $H^3(X_u,\Q)=0$ is obvious.
\end{proof}

\begin{lemma}\label{v3:b_2(X)}
\begin{equation*}  
    b_2(X)=2+\sum_{j=1}^r(n_j-1)  
\end{equation*}
\end{lemma}

\begin{proof}
Since $b_2(X)=b_4(X)$ by Poincar\'e duality, it suffices to
find~$b_4(X)$.
  
Put $X_0=p^{-1}(C\setminus S)$, $D=p^{-1}(S)$; taking into account
Lemma~\ref{v3:H^3=0}, one has the exact
sequence
\begin{multline}\label{v3:exseq}
0=H^3(D,\Q)\to H^4_c(X_0,\Q)\to H^4(X,\Q)\to H^4(D,\Q)\\ \to H^5_c(X_0,\Q)  
\to H^5(X,\Q)\to H^5(D,\Q)=0.
\end{multline}
Observe that if $D$ is a projective variety having $n$ irreducible
components \lst Dn of equal dimension~$d$, then
$H^{2d}(D,\Q)\cong\Q^d$; hence, $b_4(D)=n_1+\dots+n_r$.

Moreover, if we denote the genus of the curve~$C$ by~$g$, then
$b_5(X)=2g$. Indeed, by Poincar\'e duality it suffices to check that
$b_1(X)=2g$; now observe that, since fibers of~$p$ are swept by
rational curves, any morphism from $X$ to an Abelian variety factors
through~$p$, so the Albanese variety of~$X$ is isomorphic to that
of~$C$.

Let us find $H^4_c(X_0,\Q)$ and $H^5_c(X_0,\Q)$. To that end, let
$p_0\colon X_0\to C$ be the restriction of $p$ to $X_0$. Pick a point
$u\in C\setminus S$ and put $F=X_u$. Since $F$ is the Veronese surface
or its isomorphic projection, the ring $H^*(F,\Q)$ is generated by the
class of hyperplane section, whence $\pi_1(C\setminus S,u)$ acts
trivially on $H^*(F,\Q)$. Hence, the Leray spectral sequence of $p_0$
has the form
\begin{equation*}
E_2^{pq}=H^p_c(C\setminus S,H^q(F,\Q)) \Rightarrow H_c^{p+q}(X_0,\Q). 
\end{equation*}
Since $F\cong\PP^2$ and this spectral sequence degenerates at~$E_2$ by
Griffiths' theorem~\cite[Proposition 3.1]{Griffiths}, one finds that
$H^4_c(X_0,\Q)\cong\Q$ and $H^5_c(X_0,\Q)\cong \Q^{2g+r-1}$.

Plugging these cohomology groups of $X_0$ and~$X$ in the exact
sequence~\eqref{v3:exseq}, one obtains the result.
\end{proof}

\begin{proposition}\label{ntl:ver.pencils}
Suppose that $X$ is a Veronese pencil, $Y\subset X$ is a smooth
hyperplane section, and $m=\dim\ev(Y)$. Then the vanishing root system
of~$X$ is $A_3$ for $m=3$, it is $D_m$ for $m\ge4$, $m\ne 8$, it is
$D_8$ or $E_8$ for $m=8$, and the case $m\le2$ is impossible.
\end{proposition}

\begin{proof}
Let $i\colon Y\hookrightarrow X$ be the embedding of a general
hyperplane section~$Y$, and put $\pi=p\circ i$.

If $t\in C\setminus S$, then $\pi^{-1}(t)$ is either a smooth rational
curve of degree~$4$ (i.\,e., a transversal hyperplane section of a
smooth rational surface of degree~$4$) or the union of two different
irreducible conics that intersect transversally at one point. If
$T=\{\lst tm\}\subset C$ is the set of points with the latter
property, put $\pi^{-1}(t_j)=E_j\cup F_j$, where $E_j$ and $F_j$ are
the above-mentioned conics.

If $t=u_j\in S$, then $\pi^{-1}(t)$ is a curve with $n_j$ components;
each of these components is a smooth rational curve and their
intersection graph is a tree.

 It is clear that the self-intersection index of each $E_j$ and~$F_j$
 in~$Y$ equals~$-1$. Blowing down all the $E_j$'s one obtains a
 morphism $\sigma\colon Y\to Y'$, where $Y'$ is a smooth surface
 possessing a morphism $\pi'\colon Y'\to C$ such that ${\pi'}^{-1}(t)$
 is a smooth rational curve for $t\notin S$ and $\pi^{-1}(t)$ is a
 connected reducible curve of arithmetic genus zero with $n_j$
 rational components if $t=u_j\in S$. Arguing  as in the proof of
Proposition~\ref{v3:b_2(X)}, one sees that
\begin{equation}\label{eq:v2:X,Y}
    b_2(Y)=2+\sum_{j=1}^r(n_j-1)+m,
\end{equation}
where $m$ is the cardinality of $T\subset C$. 
Now it follows from~\eqref{eq:v2:X,Y}
and Lemma~\ref{v3:b_2(X)} that
\begin{equation}\label{v2:eq:rank(Ker)}
\rank\Ker (H_2(Y,\Z)\xrightarrow{i_*}H_2(X,\Z))=m.  
\end{equation}

Let $f\in H_2(Y,\Z)$ be the
class of $\sigma^*F'$, where $F'\subset Y'$ is a smooth fiber of
$\pi'$, let $l_j\in H_2(Y,\Z)$ be the class of $E_j$, $1\le j\le m$.
Put $L_1=\Ker (H_2(Y,\Z)\to H_2(X,\Z))$ and
\begin{equation*}
L=\{bf+\sum_{j=1}^m c_jl_j\colon c_1+\ldots+c_m=-2b\}.
\end{equation*}
Observe that, if for $t_j\in T$ one has $\pi^{-1}(t_j)=E_j\cup F_j$, then
$[E_j]=[F_j]$ in $H_2(p^{-1}(t_j),\Z)$.
Besides, any
smooth fiber of $\pi$ is algebraically equivalent to $E_j+F_j$. Hence,
$i_*(f)=2i_*(l_j)$ for each $j$, so $L\subset L_1$ and $\rank L=m$;
\eqref{v2:eq:rank(Ker)} implies that the lattice $L$ has finite index
in~$L_1$.

The same observation as in the proof of
Proposition~\ref{NTL:ord-quadr-pencils} shows that
$[L^\vee:L]=4$. Since $L\subset L_1$, it follows from Remark~\ref{newnote}
that $[L_1^\vee:L_1]$ equals~$4$ or~$1$. If, for a simply laced
irreducible root system~$R$ of rank~$m$, the index of the root lattice
in the weight lattice is~$4$, then either $m=3$ and $R=A_3$ or $m\ge 4$
and $R=D_m$; if this index is~$1$, then $m=8$ and $R=E_8$. This
completes the proof.
\end{proof}

\subsection{Del Pezzo threefolds}\label{ntl:sec:Fano2}
In this section we account for item~(4) in
Table~\ref{table:Sommese's.list}. The complete list of such varieties
is well known (see~\cite{IskIzv,Iskovskikh,Fujita,Fujita2}), and the
Betti numbers of these varieties are well known, too, which already
allows one to extract from the list those having $1$-dimensional space
of vanishing cycles. For the sake of completeness we compute the
vanishing root systems as well.

\begin{proposition}\label{ntl:DelPezzo3folds}
If $X\subset\PP^N$ is an embedded smooth Del Pezzo threefold, then it
vanishing root system is as indicated 
in Table~\ref{table:Fano2}. 
\end{proposition}

\begin{note}
As Manin indicated in~\cite[Section 23.13]{Manin}, the monodromy group
for the smooth cubic $V_3\subset\PP^4$ was first computed by Todd
in~\cite{Todd}.
\end{note}

\begin{table}
  \caption{Smooth embedded Del Pezzo threefolds}\label{table:Fano2}
  \begin{tabular}{|l|p{.5\textwidth}|c|}
    \hline
    Variety & Description & Root system\\
    \hline
    $V_3\subset \PP^4$ & Cubic & $E_6$\\
    \hline
    $V_4\subset\PP^5$& Complete intersection of two quadrics & $D_5$\\
    \hline
   $V_5\subset\PP^6$ & Codimension~$3$ linear section of the
    Grassmannian $G(2,5)\subset\PP^9$ in the Pl\"ucker embedding &
    $A_4$\\
    \hline
    $V_6\subset\PP^7$ & Segre variety $\PP^1\times\PP^1\times\PP^1$ &
    $A_1$\\
    \hline
    $V'_6\subset\PP^7$ & $\PP^*(\T_{\PP^2})$, where $\T_{\PP^2}$ is
    the tangent bundle, embedded by the complete
    linear system $|\Oo_{V'_6|\PP^2}(1)|$  & $A_2$\\
    \hline
    $V_7\subset\PP^8$ & $\PP^*(\Oo_{\PP^2}(1)\oplus \Oo_{\PP^2}(2))$
    embedded by the complete 
    linear system $|\Oo_{V_7|\PP^2}(1)|$, or an isomorphic projection
    of this variety  & $A_1$\\
    \hline
    $V_8\subset\PP^9$& Veronese variety $v_2(\PP^3)$  or an isomorphic
    projection     of this variety& $A_1$\\
    \hline
  \end{tabular}
\end{table}

\begin{proof}[Proof of Proposition~\ref{ntl:DelPezzo3folds}]
It is obvious that the vanishing root system is $A_1$ for the
variety~$V_8$ in Table~\ref{table:Fano2}. The results about vanishing root
systems of $V'_6$ and~$V_7$ follow from
Proposition~\ref{NTL:P^1-bundles}, and the result about the vanishing
root system of~$V_6$ follows from Proposition~\ref{PoQ:A_1}. It
remains to find this root system for $V_3$, $V_4$, and~$V_5$. 

If $3\le m\le 5$, then it follows from the Lefschetz hyperplane
theorem that $H_2(V_m,\Z)\cong \Z$, that the class of any line
$\ell\subset V_m$ is a generator if $H_2(V_m,\Z)$, and that if
$C\subset V_m$ is a curve, then $[C]=\deg C\cdot[\ell]\in H_2(V_m,\Z)$
(brackets mean ``the homology class'').

If $Y\subset V_m$ is a smooth hyperplane section, then $Y$ is isomorphic
to $\PP^2$ blown up at $9-m$ points; let $\sigma\colon Y\to\PP^2$ be
the corresponding blowdown. If \lst\ell {9-m} are the
exceptional curves of~$\sigma$, which are lines on $Y$, let
$l_i\in H_2(Y,\Z)$ be the class of $\ell_i$. If $H\subset\PP^2$ is a line
and $h\in H_2(Y,\Z)$ is the
class of $\sigma^*H$, then $H_2(Y,\Z)$ is the free abelian
group with basis $(h,\lst l{9-m})$. If the line $H\subset\PP^2$ does not
pass through the $9-m$ points that are blown up, then
$\sigma^{-1}(H)\subset Y\subset\PP^N$ is a twisted cubic. Hence, if
$i\colon Y\to X$ is the embedding, one has $i_*[h]=3i_*[\ell_j]\in H_2(V_m,\Z)$
for $1\le j\le 9-m$. Therefore,
\begin{multline*}
  \Ker(i_*\colon H_2(Y,\Z)\to H_2(V_m,\Z))\\
    {}=\{ah+c_1l_1+\ldots+c_{r}l_{r}\colon
c_1+\ldots+c_{r}+3a=0\},  
\end{multline*}
where $r=9-m$. 
Put $\tilde l_i=l_i-\frac13h$.
Then the root lattice
$L\subset H_2(Y,Q)$ of the vanishing root system~$R$, with the pairing
$(\cdot,\cdot)$ on $H_2(Y,\Q)$ from Section~\ref{subseq:pairings}, is
isomorphic to
\begin{equation*}
L=\left\{c_1\tilde{l}_1+\ldots +c_{r}\tilde{l}_r\colon c_j\in\Z,\ \sum c_j\equiv
0\pmod 3\right\}. 
\end{equation*}
If we put $\lambda_1=3\tilde l_1$ and $\lambda_j=\tilde l_1-\tilde l_j$ for
$2\le j\le r$, then $(\lst\lambda r)$ is a \Z-basis
of the lattice~$L$. Since $(l_i,l_j)=\delta_{ij}$, $(h,l_j)=0$ and
$(h,h)=-1$ (recall that the pairing $(\cdot,\cdot)$ from
Section~\ref{subseq:pairings} is the negated intersection index), one
has
\begin{equation*}
(\lambda_j,\lambda_k)=
  \begin{cases}
    8,&j=k=1,\\
    3,&j=1,\ k>1,\\
    2,&j=k>1,\\
    1&j>1,\ k>1,\ j\ne k.
  \end{cases}
\end{equation*}
An easy induction shows that 
\begin{equation}\label{eq:det}
\det
  \begin{vmatrix}
a&b&b&\ldots&b\\
b&c&d&\ldots&d\\
b&d&c&\ldots&d\\
\hdotsfor5\\
b&d&\ldots&d&c    
  \end{vmatrix}=(c-d)^{r-2}(a(c+(r-2)d)-(r-1)b^2),
\end{equation}
where the matrix in the left-hand side is $(r\times r)$. 
Substituting
$(a,b,c,d)=(8,3,2,1)$ in~\eqref{eq:det}, one obtains
$[L^\vee:L]=\det\|(\lambda_j,\lambda_k)\|_{1\le j,k\le m}=9-r=m$.
(Actually, we need only the cases $r=9-m$, where $m=3,4,5$, and the
three corresponding determinants can be evaluated by hand.)  It
remains to observe that if $R$ is a simply laced irreducible root
system of rank $9-m$, $3\le m\le5$, for which the index of the root
lattice in the weight lattice equals $m$, then $R=E_6$ for $m=3$,
$R=D_5$ for $m=4$, and $R=A_4$ for $m=5$.
\end{proof}

\begin{note}
In Chapter~IV of his book~\cite{Manin}, Manin associated a root system
of rank~$r=9-m$ to any Del Pezzo surface of degree~$m\le6$. In this
remark we compare these results with
Proposition~\ref{ntl:DelPezzo3folds}.

If $X\subset\PP^N$ is a Del Pezzo threefold and $Y\subset X$ is its
smooth hyperplane section, then $Y$ is a Del Pezzo surface. Since
$h^1(\Oo_Y)=h^2(\Oo_Y)=0$, one can identify $\Pic(Y)$ with
$H^2(Y,\Z)$, and one can identify $H^2(Y,\Z)$ with $H_2(Y,\Z)$ via
Poincar\'e duality. Since $\Ev(Y)$ is the orthogonal complement to
$i^*H^2(X,\Q)$, where $i\colon Y\hookrightarrow X$ is the embedding,
and since $i^*\Oo_X(1)=\omega_Y^{-1}$, we see that, if one uses the
identification above, all the elements of the vanishing root system
$R\subset \ev(Y)$ are identified with elements $l\in\Pic(Y)\cong
H^2(Y,\Z)$ such that $(l,\omega_Y)=0$ and $(l,l)=-2$ (here, we mean by
$(\cdot,\cdot)$ the standard intersection index on $\Pic(Y)$, without
negation). Thus, the vanishing root system is contained in Manin's
root system $R_r$. Observe that if $3\le m\le 5$, then the vanishing
root system \emph{equals} Manin's $R_r$. Indeed, Theorem~25.4
of~\cite{Manin} asserts that $R_r$ is $E_6$ if $r=6$, $D_5$ if $r=5$,
and $A_4$ if $r=4$. As we observed above, vanishing root system $R$ is
contained in $R_r$, and Proposition~\ref{ntl:DelPezzo3folds} implies
that the number of roots in $R$ and~$R_r$ is the same. Thus, $R$
coincides with Manin's $R_r=R_{9-m}$ if $3\le m\le5$.

For both Del Pezzo threefolds~$X$ of degree~$6$, the vanishing root
system differs from Manin's root system $A_1\times A_2$. The reason is
that in this case $i^*H^2(X,\Q)$, to which all the roots of the
vanishing root system must be orthogonal, is not spanned by the class
of $\omega_Y$, so the space of vanishing cycles is smaller than the
orthogonal complement to the class of~$\omega_Y$.
\end{note}

\subsection{Odds and ends}\label{sec:odds&ends}
Going through Sommese's list, it remains to consider the varieties in
categories~(5), (6), and (7) (Table~\ref{table:Sommese's.list}).

For the smooth quadric in~$\PP^4$ (item~(5)) it is obvious
that the vanishing root system is~$A_1$. Now we account for
items~(6) and (7).

\begin{proposition}\label{prop:odds&ends}
If $X\subset\PP^N$ is a variety in the category~\textup{(6)}
or~\textup{(7)} of Sommese's list
\textup(Table~\ref{table:Sommese's.list}\textup), then the vanishing
root system of~$X$ is~$D_5$.
\end{proposition}

\begin{proof}
Suppose first that $(X,\Oo_X(1))\cong (Q,\Oo_Q(2))$, where $Q$ is
the smooth three-dimensional quadric.  Let $Y\subset X$ be a smooth hyperplane
section. Then $(Y,\Oo_Y(1))\cong (F,\Oo_F(-2K_F))$, where $F$ is a Del
Pezzo surface of degree~$4$; we identify $Y$ with~$F$.

Let $\sigma\colon F\to\PP^2$ be a standard blowdown of five
exceptional curves, let $\lst C5\subset F$ be these curves (observe
that they are conics on $Y\subset X$), and let $l_i$, $1\le i\le 5$,
be the class of $C_i$ in $H_2(F,\Z)$. If $L\subset\PP^2$ is a line and
$h\in H_2(F,\Z)$ is the class of $\sigma^*L$, then
$H_2(Y,\Z)=H_2(F,\Z)$ is the free abelian group with the basis
$(h,\lst l5)$.

Since $X$ is isomorphic to the three-dimensional quadric,
$H_2(X,\Z)\cong\Z$ and this group is generated by the class of any
conic $C\subset X$; if $C'\subset X$ is an arbitrary curve, then
$[C']=(\deg C'/2)[C]\in H_2(X,\Z)$ (brackets mean ``the homology
class''). In particular, if $i_*\colon H_2(Y,\Z)\to H_2(X,\Z)$ is the
canonical surjection, then $i_*(h)=3[C]$. Hence, 
\begin{equation*}
  \Ker i_*=\{ah+c_1l_1+\ldots+c_5l_5\colon
c_1+\ldots+c_5+3a=0\},  
\end{equation*}
where $(l_i,l_j)=\delta_{ij}$, $(l_i,h)=0$, and $(h,h)=-1$. In the
proof of Proposition~\ref{ntl:DelPezzo3folds} we showed that this
implies that the vanishing root system is~$D_5$.

Now suppose that a variety~$X'$ belongs to the category~(7) in
Table~\ref{table:Sommese's.list}; then $X'$ is isomorphic to the
blowup of $Q$ at $k$ points \lst ak. 
If $Y'\subset X'$ is a smooth hyperplane
section, then $Y'$ is isomorphic to the blowup of $Y$ at~\lst ak, where
$Y\subset X$ is a smooth hyperplane section of~$X$ containing~\lst ak. Let
$\sigma\colon X'\to X$, $\tau\colon Y'\to Y$ be the corresponding
blowdown morphisms.

If $E_j=\sigma^{-1}(a_j)\subset X'$, $1\le j\le k$, are the
exceptional divisors and $\ell_j=\tau^{-1}(a_j)\subset Y'$, then
\begin{align*}
H_2(X',\Z)&\cong H_2(X,\Z)\oplus \Z[\ell_1]\oplus\dots\oplus\Z[\ell_k],\\
 H_2(Y',\Z)&\cong H_2(Y,\Z)\oplus \Z[\ell_1']\oplus\dots\oplus\Z[\ell_k'],
\end{align*}
where we denoted by $[\ell_j]$ and
$[\ell_j']$ the class of $\ell_j$ in $H_2(X',\Z)$ and in
$H_2(Y',\Z)$. 

Since $\sigma_*([\ell_j])=0$,
$\tau_*([\ell_j'])=0$ and $i'_*[\ell_j']=[\ell_j]$ for all~$j$, where
$i'\colon Y'\hookrightarrow X'$ is the embedding, it is clear from
the diagram
\begin{equation*}
  \xymatrix{
{Y'}\ar@{^{(}->}[r]^{i'}\ar[d]^\tau&{X'}\ar[d]^\sigma\\
Y\ar@{^{(}->}[r]^i&X
}
\end{equation*}
that
\begin{equation*}
\Ker(i'_*\colon H_2(Y',\Z)\to H_2(X',\Z))\cong  
\Ker(i_*\colon H_2(Y,\Z)\to H_2(X,\Z))
\end{equation*}
and that this isomorphism respects the from~$(\cdot,\cdot)$. Thus, the
vanishing root system of $X'$ is isomorphic to that of~$X$.
\end{proof}

Putting together Propositions~\ref{sommeses_list}, \ref{c_2=2},
\ref{PoQ:A_1}, \ref{ntl:ver.pencils}, \ref{ntl:DelPezzo3folds},
and~\ref{prop:odds&ends}, one obtains a proof of
Theorem~\ref{3folds}.

\end{document}